\documentclass[11pt]{amsart}

\usepackage{latexsym}
\usepackage{amsmath}
\usepackage{amssymb}
\usepackage{amsthm}
\usepackage[mathcal]{eucal}
\usepackage{amsmath, amssymb,bbm,enumerate}
\usepackage{mathrsfs}
\usepackage{enumerate}
\usepackage{color}

\usepackage[normalem]{ulem}
\allowdisplaybreaks

 \makeatletter \@addtoreset{equation}{section}

\makeatother \makeatletter

\newtheorem{lemma}{Lemma}[section]

\newtheorem{theorem}[lemma]{Theorem}
\newtheorem{remark}[lemma]{Remark}

\newtheorem{proposition}[lemma]{Proposition}
\newtheorem{hyp}[lemma]{Hypotheses}

\newcommand{\R}{\mathbb{R}}

\newcommand{\HH}{\mathbb{H}}

\newcommand{\N}{\mathbb{N}}

\newcommand{\Deltah}{\varDelta_{\mathbb{H}}}
\newcommand{\nablah}{\nabla_{\mathbb{H}}}

\allowdisplaybreaks

\begin{document}
\title[Singular potentials on Heisenberg groups]{Instantaneous blowup and singular potentials on Heisenberg groups}

\author[G.R. Goldstein]{Gisele R. Goldstein}
\thanks{}
\address{Department of Mathematical Sciences, University of Memphis, 38152 Memphis,
TN, USA.}
\email{ggoldste@memphis.edu}
\author[J.A. Goldstein]{Jerome A. Goldstein}
\thanks{}
\address{Department of Mathematical Sciences, University of Memphis, 38152 Memphis,
TN, USA.}
\email{jgoldste@memphis.edu}
\author[A.E. Kogoj]{Alessia E. Kogoj}
\address{Dipartimento di Scienze Pure e Applicate, Universit\`a degli Studi di Urbino Carlo Bo, Piazza della Repubblica, 13,
 I 61029 Urbino (PU), Italy.}
\email{alessia.kogoj@uniurb.it}

\author[A. Rhandi]{Abdelaziz Rhandi}

\address{Dipartimento di Ingegneria dell'Informazione, Ingegneria Elettrica e Matematica Applicata, Universit\`a degli
Studi di Salerno, Via Giovanni Paolo II, 132, I 84084 Fisciano (SA), Italy.}
\email{arhandi@unisa.it}
\author[C. Tacelli]{Cristian Tacelli}
\address{Dipartimento di Matematica, Universit\`a degli
Studi di Salerno, Via Giovanni Paolo II, 132, I 84084 Fisciano (SA), Italy.}
\email{ctacelli@unisa.it}

\begin{abstract}
In this paper we generalize the instantaneous blowup result from \cite{BG84} and \cite{GZ1} to the heat equation perturbed by singular potentials on the Heisenberg group.
\end{abstract}

\maketitle

\section{Introduction}

The problem of existence and nonexistence of nonnegative solutions to the heat equation with singular potentials  $V_c^*(x)=\dfrac{c}{|x|^2},\,x\in \Omega_N$,
\begin{equation}\label{eq1}
\left\{\begin{array}{lll}
\dfrac{\partial u}{\partial t}(x,t)=\varDelta u(x,t)+V_c^*(x)u(x,t)\quad (x,t)\in\Omega_N\times (0,\infty),\\  \\
u(x,0)=u_0(x),\quad x\in \Omega_N ,
\end{array}
\right.
\end{equation}
where $\Omega_N=\left\{ \begin{array}{ll}
\R^N\,\hbox{\ if }N\ge 2,\\
(0,\infty)\,\hbox{\ if }N=1,
\end{array}
\right.$ was settled and solved by Baras and Goldstein \cite{BG84}. For $\Omega_1=(0,\infty)$ one has to add a Dirichlet boundary condition at $0$. For simplicity we assume in the sequel that $N\ge 3$ and set $C_*(N):=\left(\frac{N-2}{2}\right)^2$.\\

Obviously, the phenomenon
of existence and nonexistence is caused by the singular potential $V_c^*$, which is controlled by Hardy’s inequality
$$C_*(N)\int_{\R^N}\frac{|\varphi(x)|^2}{|x|^2}dx\le \int_{\R^N} |\nabla \varphi(x)|^2\,dx,\quad \forall \varphi\in C_c^\infty(\R^N),$$
together with its optimal constant $C_*(N)$. Moreover $V_c^*$ belongs to a borderline case where the strong maximum principle and Gaussian bounds fail, cf. \cite{Ar68}.\\

Let $W_n(x)=\inf\{V_c^*(x),n\}$ be the cutoff potential,
with
$c>C_*(N)$. Let $u_n$ be the unique solution of
$$\left\{ \begin{array}{ll}
\dfrac{\partial_t u_n}{\partial t}-\varDelta u_n-W_nu_n=0,\quad \hbox{\ in }{\mathbb R}^N \times (0,\infty),\\ \\
u_n(x,0)=u(x,0)=u_0(x)\ge 0.
\end{array}
\right.
$$
Here $0\not\equiv u_0\in L^2(\R^N)$ or, more generally, $u_0$ grows no
faster than $e^{|x|^{2-\varepsilon}}$ at infinity. Since $W_n$ is bounded, $u_n$ exists. If a
positive solution $u$ to (\ref{eq1}) were to exist, then $0<u_n\le u$ which is a contradiction, since
$u_n(x,t)$ tends to infinity at all spatial points and at all
positive times (see \cite[Theorem 2.2.(ii)]{BG84}). This is called {\it instantaneous blowup}.

Given nonnegative functions $u_0\in L^1(\R^N),\,0\le V\in L^1_{loc}(\R^N)$ and $0\le f\in L^1(\R^N\times (0,T))$, Baras and Goldstein \cite{BG84} considered the problem of finding a function $u$ such that
$$
(P)\quad \left\{\begin{array}{ll}
0\le u \hbox{\ on }\R^N\times (0,T),\,V(\cdot)u\in L^1_{loc}(\R^N\times (0,T)),\\
\\
\dfrac{\partial u}{\partial t}=\varDelta u+V u+f\,\hbox{\ in }\mathcal{D}'(\R^N\times (0,T)),\\
\\
\mathrm{esslim}_{t\to 0^+}\int_{\R^N}u(x,t)\psi(x)\,dx =\int_{\R^N}u_0(x)\psi(x)\,dx,\,\mbox{for all} \  \psi\in \mathcal{D}(\R^N).
\end{array}
\right.
$$
Here $\mathcal{D}(\R^N):=C_c^\infty(\R^N),\,\mathcal{D}(\R^N\times (0,T)):=C_c^\infty(\R^N\times (0,T))$ with the usual topology and $\mathcal{D}'(\R^N\times (0,T))$ the dual of $\mathcal{D}(\R^N\times (0,T))$, is the space of all distributions on $\R^N\times (0,T)$.
\\
Consider the potential
$$W_0(x)=\left\{\begin{array}{ll}
 \dfrac{c}{|x|^2}\quad \hbox{\ if }x\in B_1,\\ \\
 0\quad \quad \hbox{\ if }x\in \R^N\setminus B_1.
 \end{array}
 \right.$$
 Here  $B_1$ can be replaced by  $B_\delta$ for every fixed $\delta>0$, where $B_r$ denotes the ball in $\R^N$ of center $0$ and radius $r>0$.
Baras and Goldstein \cite{BG84} proved the following result.
\begin{theorem}\label{BG}\mbox{}
\begin{itemize}
 \item [(i)] Let $0\leq c\leq C_*(N)$ and let $V\geq 0$ be a measurable potential satisfying $V\in L^{\infty}(\R^N\setminus B_1)$, where $B_1$ denotes the unit ball in $\R^N$. Let $0\le f \in L^1(\R^N\times (0,T))$.  If $V\leq W_0$ in $B_1$,
 then $(P)$ has a positive solution if
 \begin{equation}\label{eq2}
  \int_{\R^N} |x|^{-\alpha}u_0(x)\,dx<\infty\,,\quad \int_0^T\int_{\R^N} f(x,s)|x|^{-\alpha}dxds<\infty ,
 \end{equation}
where $\alpha$ is the smallest root of $\alpha(N-2-\alpha)=c$.
If $V\geq W_0$ in $B_1$, and if $(P)$ has a solution u, then
\[
 \int_{\Omega'} |x|^{-\alpha}u_0(x)\,dx<\infty,\quad \int_0^{T-\varepsilon }\int_{\Omega '}f(x,s)|x|^{-\alpha}dxds<\infty,
\]
for each $\varepsilon \in (0,T)$ and each $\Omega'\subset\subset \R^N$ with $\alpha$ as above.
If either $u_0\neq 0$ or $f\neq 0$ in $\R^N\times (0,\varepsilon)$ for each $\varepsilon\in (0,T)$, then
given $\Omega '\subset\subset \R^N$, there is a $C=C(\varepsilon, \Omega')>0$ such that
\begin{equation}\label{eq3}
 u(x,t)\geq \frac{C}{|x|^\alpha} \text{ if } (x,t)\in \Omega'\times [\varepsilon,T).
\end{equation}
\item[(ii)] If $c>C_*(N)$, $V\geq W_0$ and either $u_0 \not\equiv 0$ or $f\not\equiv 0$, then $(P)$ does not have a positive solution.
\end{itemize}
\end{theorem}

Many extensions of the above result have been done by several authors, cf. \cite{CGRT}, \cite{DR}, \cite{FR}, \cite{FGR}, \cite{GGR}, \cite{GHR}, \cite{GK}, \cite{GZ1}, \cite{GZ2}, \cite{HR}.
In this article we present a new result of this type replacing the Laplacian on $\R^N$ by the sub-Laplacian $\Deltah$ (also known as the Kohn Laplacian) on the Heisenberg group $\HH^N$. For the definitions see Section \ref{Hei-section}.

For this purpose let us
consider, for $w= (z,l) \in \HH ^N$, the problem
\begin{equation}\label{eq:pb1}
 \left\{
 \begin{array}{ll}
  \dfrac{\partial u}{\partial t}(w,t)=\Deltah u(w,t)+V_*(w)u(w,t)+f(w,t)& t>0,\,w\in \HH^N,\\
  \\
  u(w,0)=u_0(w),&w\in \HH^N.
 \end{array}
 \right.
\end{equation}
Assume $u_0\geq 0$, $f\geq 0$ and as $V_*$ choose
the corresponding critical potential in the case of the Heisenberg group $\HH^N$
\[
V_*(w)=c\frac{|z|^2}{|z|^4+l^2},\quad w=(z,l)\in \HH^N.
\]

We thus look at the problem
$$
 (P)_{\HH^N}\quad
 \left\{
 \begin{array}{ll}
  \dfrac{\partial u}{\partial t}=\Deltah u+V_* u + f &\text{in }\HH^N\times (0,T),\\
  \\
  u(w,0)=u_0(w) & w\in \HH^N,
 \end{array}
 \right.
 $$
with $u_0\geq 0$ and $u_0\not\equiv 0$ a.e.
Set $V_n(w)=\min\{V_*(w) , n\}$, $f_n(w,t) =\min\{f(w,t), n\}$. Let $u_n$ be the unique nonnegative solution of
$$
 (P_n)_{\HH^N}\quad \left\{
 \begin{array}{ll}
  \dfrac{\partial u_n}{\partial t}=\Deltah u_n+V_n u_n  + f_n &\text{ in }\HH^N\times (0,T),\\
  \\
  u_n(w,0)=u_0(w)&w\in\HH^N,
 \end{array}
 \right.
$$
and assume that $u_n$ exists.
We only need to assume that the heat equation with no potential has a global positive solution 
when $u_0$ is the initial value, see \eqref{(2.5j)} and Theorem \ref{Gaussian-estim} below. It is sufficient that $u_0\in L^2_{loc}(\HH^N)$  and $u_0$ grows no faster 
than $e^{d^2(w) -\varepsilon}$ at infinity, where $d(\cdot)$ is the function given by \eqref{*}.

Let
\[
 C^*(N)=N^2.
\]
We will prove, for $u_n$ the solution of $(P_n)_{\HH^N}$, that
\begin{itemize} 
 \item [(I)] If $0< c\leq C^*(N)$, then $$\lim_{n\to \infty}u_n(w,t)=u(w,t),\quad (w,t)\in \HH^N\times (0,T),$$ exists and is a solution of $(P)_{\HH^N}$.
 \item [(II)] If $c>C^*(N)$, then \begin{equation}\label{1.5gis}  \lim_{n\to \infty}u_n(w,t)=+\infty \end{equation} for all $(w,t)\in \HH^N\times (0,T).$
\end{itemize}
The conclusion in (II), namely \eqref{1.5gis},  is known as {\it instantaneous blowup}, or  (IBU) for short.

In the existence case $(I)$, by the maximum principle for $\Deltah$,
it is clear that we can replace $V_n,V_*$ by $\tilde V_n,\tilde V_*$ where $\tilde V_n\leq V_n$, $\tilde V_*\leq V_*$ a.e. for each $n$.
Similarly, for the nonexistence result $(II)$, we can replace $V_n,V_*$ by $\tilde V_n, \tilde V_*$ where
$\tilde V_n\geq V_n$, $\tilde V_*\geq V_*$ a.e. (at least in a neighborhood of the origin).

The paper is organized as follows. In the next section we recall the definitions of the Heisenberg group $\HH^N$ and the sub-Laplacian $\Deltah$ on $\HH^N$. We also give some known properties of $\Deltah$ that we need in this paper.
In Section \ref{main-results} we state and prove the main results of this paper. In the Appendix  we prove some technical lemmas that we use in the proof of the main results.

 This paper treats the same  basic problem as did \cite{GZ1}. There, the existence part of Theorem 3.4 was proved, using a different method. But part (ii) of Theorem 3.4 is much stronger than the corresponding result of \cite{GZ1}.

In 1999, X. Cabr\'e and Y. Martel \cite{CM} gave a different approach to a more general problem. The paper \cite{BG84} treated a potential $V\geq 0$ with only one singularity, at the origin, while  \cite{CM} allowed for a much more general  potential which 
one takes to be  $0\le W \in L^1_{loc}(\R^N\setminus  \{0\}).$ In \cite{CM} the authors defined the ``generalized first eigenvalue''  of the Schr\"odinger operator $-\varDelta - W$ as 
\begin{eqnarray*}
& & \sigma_W=\inf_{u\in C_c^1(\R^N),\,\|u\|_{L^2}=1} \left\{ \int_{\R^N} (|\nabla u(x)|^2 - W(x) |u(x)|^2)\ dx \right\} \\
& & \left(\hbox{\ or } \sigma_W=\inf_{u\in C_c^1(\R^N\setminus \{0\}),\,\|u\|_{L^2}=1} \left\{ \int_{\R^N} (|\nabla u(x)|^2 - W(x) |u(x)|^2)\ dx \right\}\hbox{\ if }N\le 2\right).
\end{eqnarray*}
 
Note that for $W(x)=\frac{c}{|x|^2},\,x\in \R^N$, $\sigma_W= -\infty$ if $c> C_*(N)$ and $\sigma_W> -\infty$ if $c\le  C_*(N).$ Roughly speaking, in \cite{CM} the existence of positive solutions, when $\sigma_W>-\infty$ and for  $\sigma_W=-\infty$, was obtained; further the authors proved that there is no globally defined pointwise solution that is exponentially bounded in time. This is a much weaker conclusion than the instantaneous blowup (IBU).  
 
 In the $\HH^N$ setting, the authors in \cite{GZ1} used the method of  \cite{CM} and proved nonexistence of globally defined (in $(x,t)$) positive solutions that grow at most exponentially for $c>C^*(N).$ But the question of (IBU) remained open until now.

 \section{Notations and preliminaries}\label{Hei-section}
The Heisenberg group and its sub-Laplacian play a crucial role in several branches of harmonic analysis, complex geometry and partial differential equations (see e.g. \cite{folland_stein_1974,garofalo_lanconelli_1990,jerison_lee_1988,jerison_lee_1989}; see also the survey papers \cite{howe, semmes}).

The Heisenberg group  $\HH^N, N\in \N$, is the stratified Lie group of step two
\begin{equation}\label{heisenberggroup} \ (\R^{2N+1}, \circ, D_\lambda). \end{equation}
If we denote the generic point of $\R^{2N+1}$ by $w=(z,l)=(x,y,l)$, with $x,y\in \R^N$ and $l\in\R$, the composition law $\circ$ is defined by
$$(x,y,l)\circ (x',y',l')=(x+x',y+y', l+l' +2 (x'\cdot y-y'\cdot x)),$$ where $x\cdot y$ denotes the inner product in $\R^N$.

In \eqref{heisenberggroup}, $D_\lambda$, $\lambda>0$ denotes the anisotropic dilation
$$D_\lambda : \R^{2N+1} \longrightarrow \R^{2N+1}, D_\lambda(z,l)=(\lambda z, \lambda^2 l).$$

The family $(D_\lambda)_{\lambda>0}$ is a group of automorphisms of $\HH^N$, that is,

$$D_\lambda((z,l)\circ (z',l')) = (D_\lambda(z,l)\circ D_\lambda(z',l')) .$$

The real number $$Q:=2N + 2$$
is called the {\it homogeneous dimension} of $\HH^N$ since it appears in the formula
$$|D_\lambda(A)|=\lambda^Q |A|,$$ where $A\subseteq \R^{2N+1}$ is a Lebesgue measurable set and $|A|$ stands for the Lebesgue measure of $A$.

A basis for the Lie algebra of left invariant vector fields on $\HH^N$ is given by

$$X_j=\partial_{x_j} + 2 y_j \partial_l, \quad Y_j=\partial_{y_j} - 2 x_j \partial_l, \quad j=1,\ldots,N.$$

One easily calculates that
\begin{eqnarray}\label{heisenbergcommutators} [X_j,X_k]=  [Y_j,Y_k]= 0\  \mbox{for every} \ j,k=1,\ldots, N,\ \mbox{and}\  [X_j,Y_k]= -4\delta_{jk} \partial_l.
\end{eqnarray}
These are the canonical commutation relations of Quantum Mechanics for position and momentum, whence $\HH^N$ is called the Heisenberg group.

The subelliptic gradient is the gradient taken with respect to the horizontal directions $\nabla_{\HH}:=(X_1,\ldots, X_N,Y_1,\ldots ,Y_N)$ and the
sub-Laplacian on $\HH^N$  is $$\Deltah := \sum_{j=1}^N (X_j^2 + Y_j^2)=\nabla_{\HH}\cdot \nabla_{\HH},$$
and it can be explicitly also written as 

$$\Deltah = \varDelta_z + 4 |z|^2 \partial_l^2 + 4 \partial_l T,$$
where
$$ \varDelta_z =  \sum_{j=1}^N  (\partial_{x_j}^2 + \partial_{y_j}^2)$$
and $$T= \sum_{j=1}^N (y_j\partial_{x_j} - x_j \partial_{y_j}).$$

From \eqref{heisenbergcommutators}  it immediately follows that
$$\mathrm{rank\ } \mathrm{Lie \ }  (X_1, \ldots, X_N, Y_1, \ldots, Y_N) (z,l) = 2N+1$$
at any point $(z,l)\in \R^{2N+1}.$  Then, by a celebrated theorem of H\"ormander, $\Deltah$ is hypoelliptic, that is, every distributional solution of $\Deltah u=f$ is smooth
whenever $f$ is smooth.

The operator $\Deltah$ is left translation invariant on $\HH^N$ and $D_\lambda$-homogeneous of degree two. Moreover $\Deltah$ has a fundamental solution (with a pole at the origin)  given by
\begin{equation} \gamma(w)=c_N \left( \frac{1}{d(w)}\right)^{Q-2}\!\!\!\!\!\! =\, c_N \left( \frac{1}{d(w)} \right)^{2N}\!\!\!\!\!\!\!,\ \quad w\neq (0,0), \end{equation}
where
\begin{equation}\label{*}
d(w)= (|z|^4 + l^2)^{\frac{1}{4}}\,\hbox{\ for }w=(z,l)\in \HH^N
\end{equation}
defines the metric $\rho(w,\tilde{w}):=d(\tilde{w}^{-1}\circ w)$ on $\HH^N$, and $\tilde{w}^{-1}$ denotes the inverse of $\tilde{w}$ in the group  $\HH^N$.

In the following lemma we summarize some properties of $d$ and its gradient $\nabla_{\HH}$ which one can obtain by simple computations, see \cite[Proposition 5.4.3]{BLU}.
\begin{lemma}\label{lem-d}
For $d(w)= (|z|^4 + l^2)^{\frac{1}{4}},\,w=(z,l)\in \HH^N$, the following hold:
\begin{eqnarray}\label{eq:11-2}
|\nabla_\HH d(w)|^2 &=& |z|^2(|z|^4+l^2)^{-\frac{1}{2}},\nonumber \\
\Deltah d(w)&=&\frac{Q-1}{d(w)}|\nabla_\HH d(w)|^2,\nonumber \\
-\Deltah d^{-\alpha}(w) &=& Cd^{-\alpha}(w)\frac{|z|^2}{|z|^4+l^2}
\end{eqnarray}
for $w\in \HH^N\setminus \{(0,0)\}$, where $C:=\alpha(Q-2-\alpha)=\alpha(2N-\alpha)$. So, $\Deltah d^{-\alpha}\in L^1_{loc}(\HH^N)$ if and only if $2N-\alpha >0$.
\end{lemma}

It is known that the left translation invariance of $\Deltah$ implies  that the semigroup $e^{t\Deltah}$ is given by a right convolution
\begin{equation}\label{(2.5j)}  e^{t\Deltah}f(w)=\int_{\HH^N}f(\tilde{w})p_t(\tilde{w}^{-1}\circ w)\,d\tilde{w},\quad t>0,\,w\in \HH^N,\end{equation} 
where $(w,t)\mapsto p_t(w)$ is the fundamental solution of $\left(\frac{\partial}{\partial t}+\Deltah \right)u=0$. Hence, by hypoellipticity,  $p_t(w)$ is a $C^\infty$ function on $\HH^N\times (0,\infty)$ and $\|p_t\|_1=1$. Moreover, $p_t$ satisfies the following Gaussian estimates, cf. \cite[Theorem IV.4.2 and Theorem IV.4.3]{VSC}.
\begin{theorem}\label{Gaussian-estim}
The heat kernel $p_t$ satisfies
$$Ct^{-\frac{Q}{2}}\exp\left(-c\frac{d^2(w)}{t}\right)\le p_t(w)\le C_\varepsilon t^{-\frac{Q}{2}}\exp\left(\frac{-d^2(w)}{4(1+\varepsilon)t}\right)$$
for some positive constants $C,\,c,\,C_\varepsilon$, any $\varepsilon>0,\,w\in \HH^N$ and $t>0$.
\end{theorem}

\section{The main results}\label{main-results}
In this section we make the following hypotheses.
\begin{hyp}\label{hyp1}

\smallskip\noindent
\begin{itemize}
 \item $0\leq V\in L^1_{loc}(\HH^N);$
 \item $0\leq f\in L^1(\HH^N\times (0,T));$
 \item $0\le u_0\in L^1(\HH^N)$ (or more generally $u_0$ can be a positive finite Radon measure).
\end{itemize}
\end{hyp}

We consider the problem

\begin{equation}\label{eq:pb4}
 \left\{
 \begin{array}{ll}
 \dfrac{\partial u}{\partial t}=\Deltah u+V u+f & \text{ in }{\mathcal D}'  (\HH^N\times (0,T)), \\ 
  {\rm esslim}_{t\to 0^+}\displaystyle\int_{\HH^N}u(w,t)\psi(w)\,dw=\displaystyle\int_{\HH^N} u_0(w)\psi (w)\,dw&
  		\text{ for all }\psi\in {\mathcal D}(\HH^N), \\ 
  u\geq 0 &\text{ on } \HH^N\times(0,T), \\ \\
  Vu\in L^1_{loc}(\HH^N\times(0,T)).
 \end{array}
 \right.
\end{equation}
Here ${\mathcal D}(\HH^N)=C_c^{\infty}(\HH^N)$ (resp. $\mathcal{D}(\HH^N\times (0,T))=C_c^\infty(\HH^N\times (0,T))$)  with the usual topologies and ${\mathcal D}':={\mathcal D}'(\HH^N)$ (resp. $\mathcal{D}'_T:=\mathcal{D}'(\HH^N\times (0,T))$) is its dual space.
We also consider the approximating problem
\begin{equation}\label{eq:pb-approx}
 \left\{
 \begin{array}{ll}
 \dfrac{\partial u_n}{\partial t}=\Deltah u_n+V_n u_n+f_n & \text{ in }{\mathcal D}'_T, \\
 \\
  {\rm lim}_{t\to 0^+}\displaystyle\int_{\HH^N}u_n(w,t)\psi(w)\,dw=\displaystyle\int_{\HH^N} u_0(w)\psi (w)\,dw,&\mbox{ for all } \psi\in {\mathcal D}(\HH^N). \\
 \end{array}
 \right.
\end{equation}
Here
\begin{align*}
&f_n=\min\{f,n\},\\
&V_n\in L^{\infty},\, 0\leq V_n\leq V,\, V_n\uparrow V\,\,a.e.
\end{align*}
By the variation of parameters formula, \eqref{eq:pb-approx} has a unique bounded nonnegative solution obtained by solving the integral equation
\begin{equation}\label{eq:variation-parameters-formula}
 u_n(t)=e^{t\Deltah}u_0+\int_0^te^{(t-s)\Deltah}V_n(\cdot)u_n(s)\,ds+\int_0^te^{(t-s)\Deltah}f_n(s)\,ds,
\end{equation}
where $\left( e^{t\Deltah} \right)_{t\geq 0}$ is the semigroup generated by $\Deltah$ on $\HH^N$.
We note that $V_n$ is a bounded multiplication operator on $L^p\left( \HH^N \right)$ for all $p\in [1,+\infty)$.
Since $\{V_n\}$ is an increasing sequence, clearly $\{u_n\}$ is an increasing sequence, as well.\\
\begin{proposition}\label{pr:1}
Suppose there is a $(w_0,t_0)\in \HH^N\times(0,T)$ with $\lim_{n\to \infty}u_n(w_0,t_0)<\infty$. Then \eqref{eq:pb4}
has a nonnegative solution on $\HH^N\times(0,T_0)$ for all $0<T_0<t_0$ given by
\begin{equation}
u(w,t)=\lim_{n\to \infty}u_n(w,t)\quad a.e. \text{ in }\HH^N\times (0,T_0).
\end{equation}
Moreover, if \eqref{eq:pb4} has a nonnegative solution in $\HH^N\times(0,T)$, then $\lim_{n\to \infty}u_n(w,t)<\infty$ a.e. in $\HH^N\times(0,T)$.
\end{proposition}
\begin{proof}
Clearly, if $u\geq 0$ is a solution of \eqref{eq:pb4}, then $u_n\leq u$ for all $n$, so
$\lim_{n\to \infty}u_n(w,t)\leq u(w,t)$ a.e. in $\HH^N\times (0,T)$. This establishes the last part of the proposition.

For the main part, we start by considering
\[
 U_n=e^tu_n,\quad t>0.
\]
Then
\[
 \frac{\partial U_n}{\partial t}=\Deltah U_n+(V_n+1)U_n+e^tf_n,
\]
and, using the variation of parameters formula, we obtain
\begin{equation}\label{eq:12}
 e^{t_0}u_n(w_0,t_0)\geq \int_0^{t_0}e^s(e^{(t_0-s)\Deltah}(V_n+1)u_n(s))(w_0)\, ds,\quad (w_0,t_0)\in \HH^N\times (0,T),
\end{equation}
since $e^{\Deltah}u_0\geq 0$ and $f_n\geq 0.$
On the other hand, it follows from the Gaussian estimates in Theorem \ref{Gaussian-estim} that
\begin{eqnarray*}
& & \int_0^{t_0}e^s(e^{(t_0-s)\Deltah}(V_n+1)u_n(s))(w_0)\, ds\\
&\ge & C\int_0^{t_0}\int_{\HH^N}e^s(V_n(\tilde{w})+1)u_n(\tilde{w},s)(t_0-s)^{-\frac{Q}{2}}\exp\left(-c\frac{d^2(\tilde{w}^{-1}\circ w_0)}{t_0-s}\right)d\tilde{w} ds.
\end{eqnarray*}
So if $\Omega '\subset \subset \HH^N $ and $\varepsilon \in (0,T)$, it follows that,
for $(w_0,t_0)\in \HH^N\times (0,T)$, there is $c_0>0$ such that
\begin{equation}\label{eq:14}
c_0\int_0^{t_0-\varepsilon}\int _{\Omega '}V_n(\tilde w)u_n(\tilde w,s)\,d\tilde w ds+
  c_0\int_0^{t_0-\varepsilon}\int _{\Omega '}u_n(\tilde w,s)\,d\tilde w ds\leq e^{t_0}u_n(w_0,t_0).
\end{equation}

By our hypothesis $u_n$ increases, moreover
the right hand side of \eqref{eq:14} is clearly bounded, so by the monotone convergence theorem, $u_n\uparrow u$
and $V_nu_n\uparrow Vu$ in $L^1\left( \Omega '\times (0,t_0-\varepsilon) \right)$
and $u$ is a solution of \eqref{eq:pb4} in the sense of distributions.
\end{proof}
\begin{remark}
Notice that the solution of \eqref{eq:pb4} satisfies the integral equation

\begin{align*}
& u(w,t)=e^{t\Deltah}u_0(w)+\int_0^t e^{(t-s)\Deltah}V(w)u(w,s)\,ds\\
&\quad +\int_0^t e^{(t-s)\Deltah}f(w,s)\,d s,\qquad (w,t)\in \HH^N\times (0,t_0).
\end{align*}
Also, since $u_n(w,t)\to u(w,t)<\infty$ a.e. on $\HH^N\times (0,t_0)$, we get, using \eqref{eq:12},
$s\longmapsto (e^{(t_0-s)\Deltah}V(\cdot)u(\cdot,s))(w)\in L^1(0,t_0)$ for a.e. $w\in \HH^N$.
\end{remark}

The inverse square potential  in the Euclidean case of $x\in\R^N$ is $V^*_c(x)= \dfrac{c}{|x|^2}$ and the critical constant is the best constant 
$$C_*(N) =\left( \frac{N-2}{2}\right)^2$$ in Hardy's inequality 
$$\int_{\R^N} |\nabla u(x)|^2\ dx \geq C_*(N) \int_{\R^N} \dfrac{|u(x)|^2}{|x|^2}\ dx$$
for $u\in C^1_{c}(\R^N)$  if $N\geq 3$ and for $u\in  C^1_{c}(\R^N\setminus \{0\})$ for $N=1,2.$\\ The multiplication operator $V^*_c$ and the Laplacian both have the same scaling property, namely
$$U(\lambda)^{-1} \mathcal{L} U(\lambda) = \lambda^2  \mathcal{L}$$ for $ \mathcal{L}=V_c^*$ or $ \mathcal{L}=\varDelta$, where $U(\lambda) f(x) = \lambda^{\frac{N}{2}}  f(\lambda x)$,  for $\lambda >0,$ defines a unitary operator on  $L^2(\R^N)$.

In the case of the Heisenberg group $\HH^N$, the corresponding critical potential is $$\tilde V^*_c(w)= \dfrac{c|z|^2}{|z|^2+l^2}$$ for $w=(x,y,l)=(z,l)\in\HH^N$ and $c>0.$ The corresponding Hardy's inequality, due to Garofalo and Lanconelli (\cite{garofalo_lanconelli_1990}, see also \cite{BLU}, \cite{GZ1}), is 
$$\int_{\HH^N} |\nabla_{\HH^N} u(w)|^2\ dw \geq C^*(N) \int_{\HH^N} \tilde V^*_1(w) |u(w)|^2\ dw,$$ with the best constant being $C^*(N)=N^2$, for all $N\geq 1.$ Both $\varDelta_{\HH^N} $ and multiplication by $\tilde V^*_c$ scale in the same way. Let 
$$\tilde U (\lambda) f (z,l)=\lambda^{N+1} f (\lambda z, \lambda^2 l);$$ 
$\tilde U (\lambda)$ is unitary on $L^2(\HH^N)$ for all $\lambda>0$ and 
$$\tilde U (\lambda)^{-1} \mathcal{L} \tilde U (\lambda)= \lambda^2 \mathcal{L}$$ 
for   $\mathcal{L}=\varDelta_{\HH^N}$ or $\mathcal{L}=\tilde V^*_c$, and all $\lambda>0.$

As in the Euclidean case, the critical potential is $C^* (N) \tilde V^*_1= \tilde V^*_{C^*(N)}$ near the origin. That is, by localizing to  the unit ball 
$B_1$ in $\HH^N$ (or to $B_\rho$ for any $\rho>0$), let

\begin{equation}\label{eq:pot-crit}
V^*_0(w)=\left\{
\begin{array}{ll}
\dfrac{c|z|^2}{|z|^4+l^2}&w \in B_1\\ \\
0&w\in \HH^N\setminus B_1,
\end{array}
\right.
\end{equation}
where $B_1$ is the unit ball centered at the origin in $\HH^N$ with respect to the metric $\rho(w,w')=d(w'^{-1}\circ w),\,w,\,w'\in \HH^N$.

Finally, notice that the smallest root of
$$
\alpha(Q-2-\alpha)=c
$$
is given by
\begin{align*}
\alpha=\frac{Q-2}{2}-\sqrt{\left( \frac{Q-2}{2} \right)^2-c}=N-\sqrt{N^2-c},
\end{align*}
when $c\leq C^*(N)$.

The following theorem is the main result of this paper. It is an extension of \cite[Theorem 1.1]{GZ1} and a generalization 
of \cite{BG84}.

\begin{theorem}\label{th:1}\mbox{}\begin{itemize}
 \item [(i)] Let $0\leq c\leq C^*(N)$ and let $V\geq 0$ be a measurable potential satisfying $V\in L^{\infty}(\HH^N\setminus B_1)$. Let $0\le f \in L^1(\HH^N\times (0,T))$. 
 If $V\leq V^*_0$ in $B_1$, then \eqref{eq:pb4} has a solution if
 \begin{equation}\label{eq:cond-ext}
  \int_{\HH^N} d(w)^{-\alpha}u_0(w)\,dw<\infty\,,\quad \int_0^T\int_{\HH^N} f(w,s)d(w)^{-\alpha}dwds<\infty ,
 \end{equation}
where $\alpha$ is the smallest root of $\alpha(2N-\alpha)=c$.
If $V\geq V^*_0$ in $B_1$, and if \eqref{eq:pb4} has a solution u, then
\[
 \int_{\HH^N} d(w)^{-\alpha}u_0(w)\,dw<\infty,\quad \int_0^{T-\varepsilon }\int_{\Omega '}f(w,s)d(w)^{-\alpha}dwds<\infty
\]
for each $\varepsilon \in (0,T)$ and each $\Omega'\subset\subset \HH^N$ with $\alpha$ as above.
If either $u_0\not\equiv 0$ or $f\not\equiv 0$ in $\HH^N\times (0,\varepsilon)$ for each $\varepsilon\in (0,T)$, then
given $\Omega '\subset\subset \HH^N$ with $0\in \Omega '$, there is a constant $C=C(\varepsilon, \Omega')>0$ such that
\begin{equation}\label{eq:t3}
 u(w,t)\geq \frac{C}{d^\alpha(\omega)} ,\quad  (w,t)\in \Omega'\times [\varepsilon,T].
\end{equation}

\item[(ii)] If $c>C^*(N)$, $V\geq V^*_0$ and either $u_0 \not\equiv 0$ or $f\not\equiv 0$, then \eqref{eq:pb4} does not have a positive solution. Moreover, we have instantaneous blowup. 

\end{itemize}

\end{theorem}
\begin{proof}

\smallskip\noindent
\\
(i): 
We can consider $V\leq V^*_0\hbox{\ in } \HH^N$. Otherwise consider $V=\tilde V+B=:V\chi_{B_1}+V\chi_{B^c_1}$ with
$\tilde V\leq V^*_0$ in $\HH^N$, $B\in L^\infty(\HH^N)$ and use Proposition \ref{pr:Bddper}  in the  Appendix.

Let $\phi(w):=d(w)^{-\alpha}$, and choose a convex function $\rho\in C^2(\R)$ with $\rho(0)=\rho'(0)=0.$
Next, multiply \eqref{eq:pb-approx} by $\rho'(u_n)\phi$ and integrate over $\HH^N\times [\delta,t)$ for $0<\delta<t<T.$ Then, letting $\int$ denote $\int_{\HH^N}$, 
\begin{align*}
&\int_\delta^t\int\frac{\partial u_n}{\partial s}\rho'(u_n)\phi = \int_\delta^t \int \Deltah u_n \rho'(u_n)\phi
 +\int_\delta^t \int (V_nu_n+f_n)\rho'(u_n)\phi ,
\end{align*}
and so
\begin{align*}
 &\int \int_\delta^t\frac{\partial }{\partial s}(\rho(u_n))\phi=
      -\int_\delta^t\int \nablah u_n\cdot\nablah (\rho'(u_n)\phi )
 +\int_\delta^t \int(V_nu_n+f_n)\rho'(u_n)\phi .
\end{align*}
Then,
\begin{align*}
 \int \rho (u_n(t))\phi &
    =-\int_\delta^t\int \rho''(u_n)|\nablah u_n |^2\phi+(\nablah u_n\cdot \nablah \phi)\rho'(u_n)\\
 &\qquad+\int_\delta^t\int(V_nu_n+f_n)\rho'(u_n)\phi +\int \rho(u_n(\delta))\phi \\
   &=\int_\delta^t\int-\rho''(u_n)|\nablah u_n|^2\phi +\int_{\delta}^t \int \rho(u_n)\Deltah \phi \\
  &\qquad +\int_\delta^t \int (V_nu_n+f_n)\rho'(u_n)\phi 
  \, +\int \rho(u_n(\delta))\phi ,
\end{align*}
since $\rho'(u_n)\nablah u_n=\nablah(\rho(u_n)).$
Hence
\begin{equation}\label{eq:15}
 \int \rho (u_n(t))\phi \leq \int_\delta^t\int\rho(u_n)\Deltah \phi +\int_\delta^t\int(V_nu_n+f_n)\rho'(u_n)\phi
  +\int \rho(u_n(\delta))\phi .
\end{equation}
Replace $\rho$ in \eqref{eq:15} with the convex function $\rho_\varepsilon (r)=\sqrt{r^2+\varepsilon^2}-\varepsilon^2,\,r\ge 0,$ and let $\varepsilon \to 0$ to obtain,
by the monotone convergence theorem,
\begin{equation}\label{eq:16}
\int u_n(t) \phi \leq\int_\delta^t\int u_n\Deltah \phi +\int_\delta^t\int (V_nu_n+f_n)\phi +\int  u_n(\delta)\phi .
\end{equation}
Next we want to let $\delta\to 0.$ Notice that
\begin{equation}\label{eq:17}
 e^{\delta\Deltah}u_0 \leq u_n(\delta)=e^{\delta (\Deltah +V_n)}u_0+\int_0^\delta e^{(\delta-s)(\Deltah+V_n)}f_n(s)ds.
\end{equation}
Since  $\|V_n\|_{\infty}=:c_n<\infty ,$ it follows from the Daletskii--Trotter product  formula that
\begin{align*}
&e^{\delta(\Deltah+V_n)}u_0=\lim_{m\to \infty }\left( e^{\delta \Deltah/m}e^{\frac{\delta}{m}V_n} \right)^mu_0\\
&\quad \leq e^{\delta c_n }e^{\delta\Deltah}u_0,
\end{align*}
by the positivity of the semigroup $\{e^{\delta \Deltah}\}.$ So \eqref{eq:17} becomes
\[
 e^{\delta \Deltah }u_0
    \leq u_n(\delta)\leq e^{\delta c_n}e^{\delta\Deltah}u_0
	+\int_0^{\delta}e^{c_n (\delta-s)}e^{(\delta-s)\Deltah}f_n(s)ds,
\]
and by the contractivity of $e^{t\Deltah}$  we have
\begin{align*}
\int \left(e^{\delta \Deltah}u_0\right)\phi  \leq \int u_n(\delta)\phi
    \leq e^{\delta c_n}\int \left(e^{\delta \Deltah} u_0\right)\phi +e^{\delta c_n }\|f_n\|_{\infty}\delta \int \phi  .
\end{align*}
The strong continuity of the semigroup implies
\[
\lim_{\delta \to 0}\int \left( e^{\delta \Deltah}u_0 \right)\phi= \int \phi u_0.
\]

\noindent
Thus we have shown that
\[
 \lim_{\delta \to 0}\int u_n(\delta)\phi= \int \phi u_0.
\]
Now let $\delta\to 0$ in $\eqref{eq:16}$, using \eqref{eq:11-2}, to deduce
\begin{align*}
&\int u_n(t)\phi\leq \int_0^t\int u_n \Deltah \phi+\int_0^t\int V_nu_n\phi+\int_0^t\int f_n\phi+\int u_0\phi \\
&\quad =\int_0^t\int \left( -c\frac{|z|^2}{|z|^4+l^2} +V_n\right)u_n \phi+\int_0^t\int f_n\phi+\int u_0\phi  \\ 
&\qquad \leq \int_0^t \int f_n\phi+\int u_0\phi ,
\end{align*}
since $V_n\leq V^*_0$. It follows that if $\int _0^t\int f\phi+\int \phi u_0<\infty$,
then, by Proposition \ref{pr:1}, $u_n(w,t)\uparrow u(w,t) (<+\infty)$ as $n\to \infty$ for all $t\in (0,T]$ and a.e. $w\in \HH^N$,
which gives the first part of (i) of  the theorem.

Let us now prove the second part of (i).
The inequality \eqref{eq:t3} is proved in Lemma \ref{technical-lemma} below. On the other hand, by the first part of (i) we have that, for each $w\in \HH^N\setminus \{ 0\},$ \eqref{eq:pb4}
has a solution with $u_0=\phi^{-1}(w)\delta_w$, $f\equiv 0$ and $V=V^*_0$, where $\delta_w$ denotes the Dirac measure at $w$.
We denote this solution by $u_w$.
We define
\[
 h_w(\tilde w,t)=u_w(\tilde w,t)\phi(\tilde w)^{-1},\quad (\tilde w ,t)\in \HH^N\times (0,T],
\]
and set
$ h=u\phi^{-1}$
and $h_n=u_n\phi^{-1}$
with $u$ (respectively $u_n$) the solution of \eqref{eq:pb4} (respectively \eqref{eq:pb-approx}) obtained by Proposition \ref{pr:1}.

We now prove 
\begin{equation}\label{eq:4.3}
h(w,t)  \geq  \int_{\HH^N} h_w(\tilde w,t)\phi(\tilde w)u_0(\tilde w)\,d\tilde w
  +\int_0^t \int_{\HH^N} h_w(\tilde w,t-s)f(\tilde w,s)\phi(\tilde w)\,d\tilde w \,ds
\end{equation}
holds for $w\in \HH^N\setminus \{ 0\}$ and $t\in (0,T]$.

To this end
let $u_n$ be the solution of \eqref{eq:pb-approx},
and let $v_n$ be the solution of 
\[\left\{
 \begin{array}{l}
  \dfrac{\partial v_n}{\partial t}= \Deltah v_n+V_{0,n}v_n\,,\\ \\
  v_n(0)=\phi(w)^{-1}\delta_w,
 \end{array}
 \right.
\]
 where $V_{0,n} =\min \left\{ V^*_0,n \right\}$.
Note that by the above construction, $v_n(\tilde w,t)\uparrow u_w(\tilde w,t)$ as $n\to \infty$ for all $t\in (0,T]$ and a.e. $\tilde w\in \HH^N$.

On the other hand, we have
\begin{align*}
&\frac{\partial }{\partial s}\int_{\HH^N} u_n(\tilde w,s)v_n(\tilde w,t-s)\,d\tilde w=
  \int_{\HH^N} \left[ \frac{\partial u_n}{\partial s}(\tilde w,s)v_n(\tilde w,t-s)-u_n(\tilde w,s)\frac{\partial v_n}{\partial s}(\tilde w,t-s)  \right]d\tilde w\\
&\quad =\int_{\HH^N} \left[  v_n(\tilde w,t-s)\Deltah u_n(\tilde w,s)-\Deltah v_n(\tilde w,t-s)u_n(\tilde w,s)+f_n(\tilde w,s)v_n(\tilde w,t-s)
      \right]d\tilde w\\
&\qquad +\int_{\HH^N} \left( V_n-V_{0,n} \right)u_n(\tilde w,s)v_n(\tilde w,t-s)\,d\tilde w\\      
&\quad  \geq \int_{\HH^N} f_n(\tilde w,s)v_n(\tilde w,t-s)\,d\tilde w.
\end{align*}
Hence integrating from $\delta$ to $t-\delta$ yields

\begin{equation}\label{eq:18} \begin{array}{lll} 
 \displaystyle\int_{\HH^N}  u_n(\tilde w,t-\delta)v_n(\tilde w,\delta)\,d\tilde w \geq \\  \\ \displaystyle\int_\delta^{t-\delta} \displaystyle\int_{\HH^N} f_n(\tilde w,s)v_n(\tilde w,t-s)\,d\tilde w \,ds
    +\int_{\HH^N} u_n(\tilde w, \delta )v_n(\tilde w, t-\delta)\,d\tilde w.\end{array}
\end{equation}

Letting $\delta \to 0$ in \eqref{eq:18} and noting that, as $\delta \to 0$,
$u_n(t-\delta)\to u_n(t)$ weakly,
$v_n(t-\delta)\to v_n(t)$ weakly, $u_n(\delta)\to u_0$ weakly and
$v_n(\delta)\to\phi(w)^{-1}\delta_w $ weakly,
we get
\begin{equation}\label{eq:19}
u_n(w,t)\phi^{-1}(w)\geq \int_0^t\int_{\HH^N} f_n(\tilde w,s)v_n(\tilde w,t-s)\,d\tilde w \,ds+\int_{\HH^N} v_n(\tilde w,t)u_0(\tilde w)\,d\tilde w.
\end{equation}
Letting $n\to \infty$ in \eqref{eq:19} and noting that $u_n(w,t)\uparrow u(w,t)=h(w,t)\phi(w)$
and $v_n(\tilde w,t)\uparrow u_w(\tilde w,t)=h_w(\tilde w,t)\phi (\tilde w)$, we obtain \eqref{eq:4.3}.

Applying \eqref{eq:t3} to $u_w$ for a fixed $w\in \HH^N\setminus \{0\}$, we obtain that there exists a constant $C>0$ such that
\[
 h_w( \tilde w,t)\geq C \text{ for  }(\tilde w ,t)\in \Omega'\times [\varepsilon,T].
\]

It follows from \eqref{eq:4.3} that
\[
 h(w,t)\geq C\int_{\Omega'}\phi(\tilde w) u_0(\tilde w)\,d\tilde w+C\int_0^{T-\varepsilon}\int_{\Omega'}f(\tilde w,s)\phi(\tilde w)\,d\tilde w ds.
\]
If a solution $u$ exists, we must have $h(w,t)<\infty$ for a.e. $w\in \HH^N$ and all $t\in (0,T]$.
Thus, necessary conditions for the existence of a solution $u$ are
\[
 \int_{\Omega'}\phi (w)u_0(w)\,dw<\infty,\text{ and }  \int_0^{T-\varepsilon}\int_{\Omega'}f(w,s)\phi(w)\,dw\,ds<\infty .
\]
This completes  the proof of (i).
\\

\noindent (ii):
Let $c>C^*(N)$ and let $u\not \equiv 0$ a solution of \eqref{eq:pb4}. Then,
\[
 \frac{\partial u}{\partial t}-\Deltah u=C^*(N)\frac{|z|^2}{|z|^4+l^2}u+(c-C^*(N))\frac{|z|^2}{|z|^4+l^2}u.
\]
From part $(i)$, a solution exists only if
\[
 (c-C^*(N))\frac{|z|^2}{|z|^4+l^2}u\phi\in L^1(\Omega'\times (0,T-\varepsilon)),
\]
for $\Omega '$ any compact set in $\HH^N$ and $\varepsilon>0$.
(Here we have assumed $0\in \Omega'$.)\\
But by the preceding proof (see \eqref{eq:t3} with $\alpha=N$), we have
$$u\geq C_\varepsilon d^{-N}(w),$$
in $\Omega'\times [\varepsilon,T)$, and so we would need $\frac{|z|^2}{|z|^4+l^2}d^{-N}\in L^1(\Omega')$, which is false.

\end{proof}

\mbox{}

\begin{lemma}\label{technical-lemma}
Assume  $0\le c\le C^*(N),\,\alpha$ the smallest root of $\alpha(2N-\alpha)=c$, and $0\le V\in L^\infty(\HH^N\setminus B_1)$ with $V\ge V^*_0$ in $B_1$. If $u$ is a solution of \eqref{eq:pb4} with $u_0\not\equiv 0$ in $\HH^N$, then given $\Omega'\subset \subset \HH^N,\,\varepsilon \in (0,T)$, there is $C=C(\varepsilon ,\Omega')>0$ such that
\eqref{eq:t3} holds.
\end{lemma}
\begin{proof}
Assume that
$\Omega'\subset \subset \HH^N$ with $0\in \Omega'$.
Since $u_0\not\equiv 0$, it follows from Theorem \ref{Gaussian-estim} that there is a constant $C_0>0$ with
\begin{equation}\label{etoile}
 e^{t\Deltah}u_0(\tilde w)\geq C_0,
\end{equation}
for $\tilde w\in \Omega'$ and $\frac{\varepsilon}{2}\leq t<T$.  Since $u\ge e^{t\Deltah}u_0$, by the Maximum Principle, \eqref{eq:t3} follows from \eqref{etoile} for the case $\alpha=0$. So from now on we assume that $\alpha$ is strictly positive.

Let as before $V_{0,n}:=\inf \left\{ V^*_0,n\right\}$, and
consider the problems
\begin{equation}\label{eq:21}
\left\{
\begin{array}{ll}
\dfrac{\partial z}{\partial t}=\Deltah z+V^*_0z &\hbox{\ in }\mathcal D'\left(\HH^N\times [\frac{\varepsilon}{2},T]\right),\\
\\

z\left(\tilde w,\frac{\varepsilon}{2}\right)=C_0\,\chi_{\Omega'}(\tilde w) &\tilde w\in \HH^N,
\end{array}
\right.
\end{equation}

\begin{equation}\label{eq:22}
\left\{
\begin{array}{ll}
\dfrac{\partial z_n}{\partial t}=\Deltah z_n+V_{0,n}z_n &\hbox{\ in }\mathcal D'\left(\HH^N\times [\frac{\varepsilon}{2},T]\right),\\
\\
z_n\left(\tilde w,\frac{\varepsilon}{2}\right)=C_0\,\chi_{\Omega'}(\tilde w)&\tilde w\in\HH^N,
\end{array}
\right.
\end{equation}
and for $B_{r_0}\subset \Omega'$ a ball centered at the origin with radius $r_0\in (0,1),$
\begin{equation}\label{eq:23}
\left\{
\begin{array}{lll}
\dfrac{\partial v_n}{\partial t}=\Deltah v_n+V_{0,n}v_n &\hbox{\ in } \mathcal D'\left(B_{r_0}\times [\frac{\varepsilon}{2},T]\right),\\ \\
v_n=0 &\text{ on }\partial B_{r_0},\\ \\
v_n\left(\tilde w,\frac{\varepsilon}{2}\right)=C_0 &\tilde w\in B_{r_0}.
\end{array}
\right.
\end{equation}
Notice that \eqref{eq:22} has a unique solution $z_n\geq0$, also that $z_n(\tilde w,t)\uparrow z(\tilde w,t)$, for almost every  $(\tilde w,t)\in \HH^N\times
[\frac{\varepsilon}{2},T)$,
where $z$ is the unique solution of \eqref{eq:21}.
It is also clear that $z_n\geq v_n$, the solution of \eqref{eq:23},
and that $v_n$ is a radial function\footnote{
Recall that a function $g(w)$ is radial on $\HH^n$ if $w=(z,l)$ and $g(z,l)=g(|z|,l)$.
In fact, our function $v_n$ is even more special, since $v_n=v_n(d(w)).$ Notice that this gives $\nablah v_n=v'_n(d(w))\nablah d(w)$.
}.
Finally, we note that $u$ is bounded below by the solution of \eqref{eq:21} since $V\ge V^*_0$.

Multiply the equation in $\eqref{eq:23}$ by $v_n^{p-1}\phi^{2-p},\,p\geq 2,$  where we recall that $\phi(w)=d(w)^{-\alpha}$ only depends on $w$, and  integrate to get
\[
 \int_{B_{r_0}}\frac{\partial v_n}{\partial t}v_n^{p-1}\phi^{2-p}=\int_{B_{r_0}}(\Deltah v_n)v_n^{p-1}\phi^{2-p}+\int_{B_{r_0}}V_{0,n}v_n^p\phi^{2-p},
\]
so
\begin{equation}\label{eq:24-0}
\frac{\partial}{\partial t}\int_{B_{r_0}}\frac{1}{p}\left( \frac{v_n}{\phi} \right)^p\phi^2=
  -\int_{B_{r_0}}\nablah v_n\cdot\nablah\left( v_n^{p-1}\phi^{2-p} \right)+\int_{B_{r_0}}V_{0,n}\left( \frac{v_n}{\phi} \right)^p\phi^2.
\end{equation}

Set $g_n=\frac{v_n}{\phi}$. Then equation \eqref{eq:24-0} becomes
\begin{align*}
\frac{\partial }{\partial t}\int_{B_{r_0}}\frac{1}{p}g_n^p\phi^2=-\frac{4(p-1)}{p^2}\int_{B_{r_0}}|\nablah g_n^{p/2}|^2\phi^2
    +\int_{B_{r_0}}g_n^p\left(\Deltah \phi\right) \phi+\int_{B_{r_0}}V_{0,n}g_n^p\phi^2.
\end{align*}

Using \eqref{eq:11-2} and the fact that $\alpha(2N-\alpha)=c$,  we obtain $V_{0,n}\leq V^*_0\leq -\frac{\Deltah \phi}{\phi}$, so that
$V_{0,n}\phi^2\leq \left( -\Deltah \phi \right)\phi$.
Hence we have shown
\begin{equation*}
 \frac{\partial}{\partial t}\int_{B_{r_0}}g_n^p\phi^2\leq 0,
\end{equation*}
and we thus have, for $\frac{\varepsilon}{2}\leq t\leq T$, that
\begin{equation}\label{eq:24}
\left( \int_{B_{r_0}}v_n^{p}\phi^{2-p} \right)^{1/p}\leq C_0\left( \int_{B_{r_0}}\phi ^{2-p}\right)^{1/p}.
\end{equation}
Letting $p\to \infty$ in \eqref{eq:24} we get
\begin{equation}\label{eq:g}
g_n\leq C_0 \hbox{\ a.e. in }B_{r_0},
\end{equation}
which is equivalent to $v_n\leq C_0\phi$ a.e. in $B_{r_0}$.
So we can make sense of
\[
 v=\lim_{n\to \infty}v_n \text{ and }g=\lim_{n\to \infty} g_n.
\]
Now we claim that
\begin{equation}\label{eq:25}
 0<C_1\le g(w,t)\leq C_0
\end{equation}
for $t\in [\varepsilon,T]$ and a.e. $w\in B_{\frac{r_0}{2}}$.
Once \eqref{eq:25} is proved, and since
\begin{equation*}
 u\geq z\geq z_n\geq v_n=g_n\phi,
\end{equation*}
 \eqref{eq:t3} follows directly in the case $\Omega '=B_{\frac{r_0}{2}}$. Otherwise,
we observe that for almost every  $\tilde w\in\Omega'\setminus B_{\frac{r_0}{2}}$ we have
\begin{equation*}
 h(\tilde w,t)=\phi (\tilde w)^{-1}u(\tilde w)\geq \phi(\tilde w)^{-1}\left( e^{t\Deltah}u_0 \right)(\tilde w)\geq C_2>0,
\end{equation*}
from Theorem \ref{Gaussian-estim} since  
\begin{equation*}
 \phi(\tilde w)^{-1}\geq C_3>0,
\end{equation*}
for all $t\in [\varepsilon ,T]$  and some constants $C_2,C_3>0$.
This concludes the proof of \eqref{eq:t3}.

Now we must  prove \eqref{eq:25}.
By \eqref{eq:g}, the right inequality is proved. For the remaining part of \eqref{eq:25}, let $\mathcal I\in C^2(\R_+,\R_+)$ be convex. Multiply equation \eqref{eq:23}
by ${\mathcal I}'(g_n){\mathcal I}(g_n)\phi\psi^2$ and integrate over
$Q=B_{r_0}\times\left( \frac{\varepsilon}{2},t \right),\,t\in [\frac{\varepsilon}{2},T]$,
where
$\psi\in {\mathcal D}(B_{r_0}\times (\frac{\varepsilon}{2},T])$,
to get
\begin{equation*}
 \int_Q\mathcal I'(g_n)\mathcal I(g_n)\frac{\partial v_n}{\partial t} \phi\psi^2=\int_Q
 \left\{
 \Deltah v_n\mathcal I'(g_n)\mathcal I(g_n)\phi\psi^2+V_{0,n}v_n\mathcal I'(g_n)\mathcal I(g_n)\phi\psi^2
 \right\},
\end{equation*}
\begin{equation*}
 \frac{1}{2} \int_Q  \frac{\partial }{\partial t}
  ( \mathcal I (g_n))^2 \phi^2\psi^2
      =-\int_Q\nablah (g_n\phi)\cdot \nablah (\mathcal I'(g_n)\mathcal I(g_n)\phi\psi^2)
      +\int_QV_{0,n}g_n\mathcal I'(g_n)\mathcal I(g_n)\phi^2\psi^2.
\end{equation*}
Notice that
\begin{align*}
&\int_{B_{r_0}}\nablah (g_n\phi)\cdot \nablah (\mathcal I'(g_n)\mathcal I (g_n)\phi\psi^2) \\
&\quad =\int_{B_{r_0}}\left\{
      (\nablah g_n\cdot\nablah (\mathcal I'(g_n))\mathcal I(g_n)\phi\psi^2)\phi+g_n(\nablah\phi\cdot\nablah(\mathcal I'(g_n))\mathcal I(g_n)\phi\psi^2 ) \right\}\\
&\quad = \int_{B_{r_0}}\left\{
     \mathcal I''(g_n)| \nablah g_n|^2\mathcal I(g_n)\phi^2\psi^2+|\nablah \mathcal I(g_n)|^2\phi^2\psi^2   \right\}\\
&\qquad+\int_{B_{r_0}}\left\{
      (\nablah \mathcal I(g_n)\cdot \nablah \phi) \psi^2\phi \mathcal I(g_n)+(\nablah\mathcal I(g_n)\cdot\nablah\psi^2)\mathcal I(g_n)  \phi^2 \right\}\\
&\qquad+\int_{B_{r_0}} (-\Deltah \phi)g_n\mathcal I'(g_n)\mathcal I(g_n)\phi\psi^2
    +\int_{B_{r_0}}-\left( \nablah\phi\cdot \nablah\mathcal I(g_n)\right)\mathcal I(g_n)\phi\psi^2\\
&\quad = \int_{B_{r_0}}\left\{
     \mathcal I''(g_n)| \nablah g_n|^2\mathcal I(g_n)\phi^2\psi^2+|\nablah \mathcal I(g_n)|^2\phi^2\psi^2   \right\}\\
&\qquad+ \int_{B_{r_0}}(\nablah\mathcal I(g_n)\cdot\nablah\psi^2)\mathcal I(g_n)  \phi^2 +\int_{B_{r_0}} (-\Deltah \phi)g_n\mathcal I'(g_n)\mathcal I(g_n)\phi\psi^2,
\end{align*}
and so
\begin{align}\label{eq:26}
&\frac{1}{2}\int_{Q}\frac{\partial}{\partial t}\left( \mathcal I(g_n)^2 \right)\phi^2\psi^2+\int_Q\nablah \mathcal I(g_n) \cdot \nablah\psi^2\mathcal I(g_n)\phi^2
    \nonumber \\
&\quad =-\int_Q\mathcal I''(g_n)|\nablah g_n|^2\left( \mathcal I(g_n)\phi^2\psi^2 \right) \nonumber \\
&\qquad -\int_Q |\nablah\mathcal I(g_n)|^2 \phi^2\psi^2+\int \Deltah \phi g_n\mathcal I'(g_n)\mathcal I(g_n)\phi\psi^2\\
&\qquad +\int_QV_{0,n}g_n\mathcal I'(g_n)\mathcal I(g_n)\phi^2\psi^2.\nonumber
\end{align}
Using H\"older's inequality,
\begin{align*}
&\left| 2\int_{B_{r_0}} \left( \nabla \mathcal I(g_n)\cdot \nablah\psi\right) \mathcal I(g_n)\phi^2\psi
 \right|\\
&\quad \leq \frac{1}{2}\int_{B_{r_0}}|\nablah \mathcal I (g_n)|^2\phi^2\psi^2
    +2\int_{B_{r_0}}|\nablah \psi |^2 |\mathcal I(g_n)|^2 \phi^2,
\end{align*}
on the second term of the left hand side of \eqref{eq:26}, using the convexity assumption on $\mathcal I$, and integrating by parts on the first term in \eqref{eq:26}, we obtain
\begin{align}\label{eq:27}
&\frac{1}{2} \int_{B_{r_0}}\left(\mathcal I(g_n)^2 \psi^2\phi^2\right)(t)+\frac{1}{2}\int_{Q}|\nablah \mathcal I(g_n)|^2 \phi^2\psi^2\nonumber\\
&\quad \leq \int_{Q} (V_{0,n}\phi+\Deltah \phi)g_n\mathcal I '(g_n)\mathcal I (g_n)\phi\psi^2\\
&\qquad +\int_{Q}\mathcal I (g_n)^2(2|\nablah\psi|^2+\psi\frac{\partial \psi}{\partial t})\phi^2\nonumber.
\end{align}

Now, we make a key observation.
Since $\alpha <N$, we have $\phi\Deltah \phi\in L^1(B_{r_0})$. Assume $r_0$ to be sufficiently  small.
Since $V^*_0=\frac{-\Deltah \phi}{\phi}$,
the first term on the right hand side of \eqref{eq:27} converges  to $0$ as $n\to \infty$ by Lebesgue's dominated convergence theorem,
since $\|g_n\|\leq C_0$ in $B_{r_0}$ and $\mathcal I$ is convex, $C^2$ and nonnegative. Letting $n\to \infty$,  \eqref{eq:27} gives
\begin{align}\label{eq:28}
&\int_{B_{r_0}}\mathcal I(g)^2 \psi^2\phi^2+\int_{Q}|\nablah \mathcal I(g)|^2\psi^2\phi^2\nonumber\\
&\qquad \leq 2\int_Q\mathcal I(g)^2 (2 |\nablah \psi |^2+\psi\frac{\partial \psi}{\partial t})\phi ^2.
\end{align}
Choose $\psi$ so that $0\leq \psi\leq 1$

for $s\geq \frac{\varepsilon}{2}, \,r<r_0$ and $0<\delta<r$, and
\begin{equation*}
 \psi(w,t)=\left\{
 \begin{array}{l}
  1\qquad B_{r-\delta}\times[s+\delta,T],\\
  0\qquad \left( B_{r_0}\times [0,s] \right)\cup \left( B_{r_0}\setminus
  	B_{r-\frac{\delta}{2}}\times[0,T] \right),
 \end{array}
 \right.
\end{equation*}
 so that
\begin{equation*}
\left| \frac{\partial \psi}{\partial t}\right|\leq \frac{\tilde C}{\delta},\qquad \left|\nablah \psi\right|^2\leq \frac{\tilde C}{\delta^2},
\end{equation*}
for some constant $\tilde C$ independent  of $s,\delta.$ Then for all $s+\delta\leq t\leq T$ \eqref{eq:28} becomes
\begin{align}\label{eq:29}
&\int_{B_{r-\delta}} \mathcal I(g(t))^2\phi^2+\int_{s+\delta}^T\int_{B_{r-\delta}}|\nablah \mathcal I(g)|^2 \phi ^2\nonumber\\
&\quad \leq 6\tilde C\delta^{-2}\int_s^T\int_{B_{r_0}}\mathcal I(g)^2\phi^2.
\end{align}

Note that for fixed $t$, $w\mapsto \mathcal I(g(w,t))$ is a radial function; in fact as we  noted earlier, $\mathcal I(g(w,t))$ is a function
of $d(w)$. Applying \eqref{eq:32}, with $\beta$ as in Lemma \ref{lem:GGKRT}, and \eqref{eq:29}, one obtains
\begin{align*}
&\int_{s+\delta}^T\int_{B_{r-\delta}}\mathcal I(g)^{2+2\beta}\phi^2\\
&\quad  \leq \hat {\hat C}
    \int_{s+\delta}^T\left[
    \left(\int_{B_{r-\delta}} |\nablah \mathcal I(g(t))|^2 \phi ^2+\mathcal I(g(t))^2\phi^2\right)
    \left( \int_{B_{r-\delta}} \mathcal I(g(t))^2\phi^2\right)^{\beta}
        \right]dt
    \\
&\quad  \leq \hat {\hat C}
\left[ \int_{s+\delta}^T
\int_{B_{r-\delta}} |\nablah \mathcal I(g)|^2 \phi ^2
        +\int_{s+\delta}^T\int_{B_{r-\delta}}\mathcal I(g)^2\phi^2
\right]\left( 6\tilde C\delta^{-2}\int_s^T\int_{B_{r_0}}\mathcal I(g)^2\phi^2 \right)^{\beta}
    \\
&\quad      \leq \hat {\hat C} (6\tilde C\delta^{-2}+1)
      \left( \int_s^T\int_{B_{r_0}}\mathcal I(g)^2\phi^2 \right)
      \left( 6\tilde C\delta^{-2}\int_s^T\int_{B_{r_0}}\mathcal I(g)^2\phi^2 \right)^{\beta}.
\end{align*}
Since $0<\delta<r< 1$, it follows that
\begin{align}\label{eq:34}
&\left( \int_{s+\delta}^T\int_{B_{r-\delta}}\mathcal I(g)^{2+2\beta}\phi^2 \right)^{\frac{1}{2+2\beta}}
    \leq \hat{\hat C}^{\frac{1}{2+2\beta}} (6\tilde C+1)^{1/2}\delta^{-1}\left( \int_s^T \int_{B_{r_0}} \mathcal I(g)^2\phi^2\right)^{1/2}\nonumber \\
&\quad \leq  \overline C\delta^{-1}\left( \int_s^T\int_{B_{r_0}}\mathcal I(g)^2\phi^2 \right)^{1/2}.
\end{align}
Let $b>0$ be sufficiently small, and set
\begin{align*}
&\delta=\frac{b}{2^n},\quad r_{n+1}=r_n-\frac{b}{2^n},\quad \mathcal I_{n+1}=\mathcal I_{n}^{1+\beta},\quad s_{n+1}=s_n+\frac{b}{2^n},\\
&\text{ and } k_n=\left( \int_{s_n}^T\int_{B_{r_n}}\mathcal I_n(g)^2\phi^2 \right)^{1/2}.
\end{align*}
Here $\mathcal I_1=\mathcal I$ and $r_1,s_1>0$ with $s_1\geq \frac{\varepsilon}{2}$ and $r_1<1$ are given.
Applying \eqref{eq:34} yields
\begin{equation}\label{eq:35}
k_{n+1}^{\frac{1}{1+\beta}}\leq\overline C2^nb^{-1}k_n.
\end{equation}
Applying Lemma \ref{lem:kn} in the Appendix, we have
\[
k_n^{\frac{1}{(1+\beta)^{n-1}}}
    \leq \left( \frac{\overline C}{b} \right)^{ \frac{a_n}{(1+\beta)^{n-2}} }2^{ \frac{d_n}{(1+\beta)^{n-2}} }k_1,
\]
where $a_n=\sum_{j=0}^{n-2}(1+\beta)^j$ and $d_n=\sum_{j=0}^{n-2}(j+1)(1+\beta)^{n-2-j}$ for $n\ge 2$.
Letting $n\to \infty$
we have $$s_n\to s_1+b, \,r_n  \to r_1-b ,\,
\frac{a_n}{(1+\beta)^{n-2}}\to \left(\frac{1+\beta}{\beta}\right),\,
\frac{d_n}{(1+\beta)^{n-2}}\to \left(\frac{1+\beta}{\beta}\right)^2$$
and
taking into account that $\mathcal I_n=\mathcal I^{(1+\beta)^{n-1}}$ we obtain
\[
 k_n^{\frac{1}{(1+\beta)^{n-1}}}=
        \left( \int_{s_n}^T\int_{B_{r_n}}\mathcal I(g)^{2(1+\beta)^{n-1}}\phi^2 \right)^
                {\frac 1{2(1+\beta)^{n-1}}}
                    \to \sup_{B_{r_1-b}\times [s_1+b,T]} \mathcal I(g).
\]
Finally, we have
\begin{equation}\label{BG-gamma}
 \sup_{B_{r_1-b}\times [s_1+b,T]} \mathcal I(g)
                        \leq
                    \left( \frac{\overline C}{b} 2^{\frac{1+\beta}{\beta}}\right)^{\frac{1+\beta}{\beta}}
                    \left( \int_{s_1}^T\int_{B_{r_1}}\mathcal I(g)^2\phi^2 \right)^{1/2}.
\end{equation}

   Now, consider
a sequence $\mathcal I_n(r) \in C^2(\R_+,\R_+)$ of convex function
converging to $\frac{1}{r^\gamma}$ as $n\to \infty$,
where $\gamma>0$ is a parameter to be chosen later.
Replace $\mathcal I$ by a $\mathcal I_n$ in \eqref{BG-gamma}, we obtain
\[
 \sup_{B_{r_1-b}\times [s_1+b,T]}  g^{-\gamma}
                        \leq
                    \left( \frac{\overline C}{b} 2^{\frac{1+\beta}{\beta}}\right)^{\frac{1+\beta}{\beta}}
                    \left( \int_{s_1}^T\int_{B_{r_1}} g^{-2\gamma}\phi^2 \right)^{1/2}.
\]

Set $s_1= \frac{3}{4}\varepsilon$, $b=\frac{\varepsilon}{4}$ and $r_1<r_0$,
where $r_0$ is the one chosen in the beginning of the proof.
We have
\[
 g(\omega,t)=\frac{v}{\phi}\geq \phi^{-1}(w)\left( e^{t\Deltah}C_0 \right)(w)= C_0\phi^{-1}(w),
\]
 for almost every  $w\in B_{r_0}$ where $C_0$ is the constant given in \eqref{etoile}.
So
\[
 \sup_{B_{r_1-\varepsilon/4}\times [\varepsilon,T]}  g^{-\gamma}
                        \leq
                    C_2C_0^{-\gamma}
                    \varepsilon^{-\frac{1+\beta}{\beta}}
                    \left( \int_{\frac{3}{4}\varepsilon}^T\int_{B_{r_1}} \phi^{2+2\gamma} \right)^{1/2},
\]
and it follows that
\begin{equation}\label{eq:415}
 g(w,t)\geq C_2^{-\frac{1}{\gamma}}C_0\varepsilon^{\left( 1+\frac{1}{\beta} \right)\frac{1}{\gamma}}
 \left(\int_{B_{r_1}}\phi^{2+2\gamma}\right)^{-\frac{1}{2\gamma}},
\end{equation}
for almost every  $w\in B_{r_1-\frac{\varepsilon}{4}}$ and for all $t\in [\varepsilon,T]$, where $C_2$ is a positive constant independent of
$\varepsilon$ and $r_1.$
A simple computation shows that $\int_{B_{r_1}}\phi^{2+2\gamma}<\infty$ by choosing $0<\gamma <\frac{N+1}{\alpha}-1$, which is possible since $\alpha\in (0,N]$.

Thus \eqref{eq:25} follows by taking $\gamma\in (0,\frac{N+1}{\alpha}-1)$ and $\varepsilon=2(2r_1-r_0)$ with $\frac{r_0}{2}<r_1<r_0$.

This concludes the proof.
\end{proof}

We end this section by the following remark.
\begin{remark}
The arguments used are based on the explicit form of the fundamental solution and on the existence of an underlying group of dilations; thus, the results would likely extend 
to the setting of H-type groups.
\end{remark}

\appendix{}
\section{} 

In this Appendix we collect all technical lemmas that we needed for proving the main result.
\begin{lemma}\label{lem:kn}
 For $\beta>0$ and $n\geq 2$, define $a_n=\sum_{j=0}^{n-2}(1+\beta)^j$ and $d_n=\sum_{j=0}^{n-2}(j+1)(1+\beta)^{n-2-j}$ and $k_n\ge 0$ for $n\geq 1$ such that
 $$k_{n}^{\frac{1}{1+\beta}}\leq\overline C2^{n-1}b^{-1}k_{n-1}.$$ Then
\begin{equation}\label{eq:36}
k_n^{\frac{1}{1+\beta}}\leq \left( \frac{\overline C}{b} \right)^{a_n}2^{d_n}k_1^{(1+\beta)^{n-2}}.
\end{equation}
\end{lemma}
\begin{proof} We use an induction argument. Assume \eqref{eq:36}  is true for $1\leq k\leq n.$ We will show
\begin{equation*}
k_{n+1}^{\frac{1}{1+\beta}}\leq \left( \frac{\overline C}{b} \right)^{a_{n+1}}2^{d_{n+1}}k_1^{(1+\beta)^{n-1}}.
\end{equation*}
Clearly \eqref{eq:35} gives \eqref{eq:36} if $n=1.$
By \eqref{eq:35},
\begin{equation*}
 k_{n+1}^{\frac{1}{1+\beta}}\leq \overline C 2^nb^{-1}k_n
      \leq\left( \frac{\overline C}{b} \right)2^n\left( \frac{\overline C}{b} \right)^{a_n(1+\beta)}2^{d_n(1+\beta)}k_1^{(1+\beta)^{n-1}},
\end{equation*}
by the induction hypothesis. Now it is easy to check that
\begin{equation*}
 a_n(1+\beta)+1=\sum_{j=0}^{n-2}(1+\beta)^{j+1}+1=\sum_{j=0}^{n-1}(1+\beta)^j=a_{n+1}
\end{equation*}
and
\begin{align*}
d_n(1+\beta)+n=\sum_{j=0}^{n-2}(j+1)(1+\beta)^{n-1-j}+n
    =\sum_{j=0}^{n-1}(j+1)(1+\beta)^{n-1-j}=d_{n+1}.
\end{align*}
\end{proof}

The following two lemmas can be found in \cite[Appendix]{BG84}.
\begin{lemma}\label{BG-lem1-techn}
If $0\leq h\in C^1[0,2r]$ and $h(2r)=0$, then
\begin{equation}\label{eq:30}
 \left( \int_0^{2r} h^p(s)s^{\gamma-1}ds\right)^{1/p}\leq M_0\left( \int_0^{2r} \left|\frac{dh}{ds}\right|^{2}s^{\gamma -1}ds \right)^{1/2},
\end{equation}
where $\gamma=2N+2-2\alpha>2$, $\frac{1}{p}=\frac{1}{2}-\frac{1}{\gamma}$ and $M_0$ is a constant depending only on $\gamma.$
\end{lemma}

\begin{lemma}\label{BG-lem2-techn}
If $0<r'\leq r\leq 1$, $0\leq h\in C^1[0,r]$ then
\begin{equation}\label{eq:31}
 \left( \int_0^r|h(s)|^p s^{ 2N-2\alpha+1}ds\right)^{2/p}
      \leq \hat C\left( \int_0^r[|h'(s)|^2+|h(s)|^2]s^{2N-2\alpha+ 1}ds \right),
\end{equation}
where $\frac{1}{p}\geq \frac{1}{2}-\frac{1}{2N+2-2\alpha}$ and $p=\infty$ if $N=\alpha-1$, $\hat C$ depends on $r'$ but not $r.$
\end{lemma}

The following lemma is needed for the proof of Lemma \ref{technical-lemma}.
\begin{lemma}\label{lem:GGKRT}
If $k\in C^1(\R_+,\R_+),\,\tilde{k}(w):=k(d(w)),\,\phi(w)=d(w)^{-\alpha},\,w=(z,l)\in \HH^N,$ and $0<\beta $ is such that $\beta+\frac{2}{p}=1$, where $\frac{1}{p}=\frac{1}{2}-\frac{1}{2N-2\alpha +2}$, then

\begin{equation}\label{eq:32}\begin{array}{ll}
\displaystyle\int_{B_r}\tilde{k}^{2+2\beta}(w)\phi^2(w)\,dw\leq \\ \\
  \hat{\hat C} \left(\displaystyle\int_{B_r}\left( |\nablah \tilde{k}(w)|^2+\tilde{k}^2(w) \right)\phi^2(w)\,dw\right)\left( \displaystyle \int_{B_r}\tilde{k}^2(w)\phi^2(w)\,dw \right)^\beta.\end{array} 
\end{equation}

\end{lemma}
\begin{proof}
We first  prove that
\begin{equation}\label{eq:finale}
\int_{B_{r}}|\nablah k(d(\omega))|^2 \phi ^2(d(w))\,dw
=C_N\int_0^{r} \left|k'(s))\right|^2s^{2N-2\alpha +1}\,ds,
\end{equation}
where $C_N:=S_{2N-1} \int_{-\frac{\pi}{2}}^{\frac{\pi}{2}}\cos^N\varphi \,d\varphi$ and 
$S_{2N-1}$ the surface area of the unit ball in $\R^{2N}$.

Consider the change of variables $\rho =|z|$ in  polar coordinates, and take
$$\left\{\begin{array}{ll}
\rho^2=r^2\cos \varphi, \\
l=r^2\sin \varphi
\end{array}
\right.
$$
with $\varphi \in [-\frac{\pi}{2},\frac{\pi}{2}]$. Recalling from Lemma \ref{lem-d} that $|\nablah d(w)|^2=|z|^2d(w)^{-2},\quad w=(z,l)\in \HH ^N$ and since $\nablah k(d(w))=k'(d(w))\nablah d(w)$ we obtain
\begin{eqnarray*}
& & \int_{B_r} |\nablah k(d(w)) |^2 \phi ^2(d(w))\,dw \\
&=& S_{2N-1} \int |\nablah k (\sqrt[4]{\rho^4+l^2})|^2\phi ^2(\sqrt[4]{\rho^4+l^2})\rho^{2N-1}\,d\rho dl \\
&=& S_{2N-1} \int_{-\frac{\pi}{2}}^{\frac{\pi}{2}}\int_0^r | k'(s)|^2(s^2\cos \varphi) s^{-2}s^{-2\alpha}(s\sqrt{\cos \varphi})^{2N-1}
    \left(\frac{s^2}{\sqrt{\cos \varphi}}\right)\,ds d\varphi\\
&=& C_N \int_0^r |k'(s)|^2 s^{2N-2\alpha +1}\,ds
\end{eqnarray*}
for $C_N:=S_{2N-1} \int_{-\frac{\pi}{2}}^{\frac{\pi}{2}}\cos^N\varphi \,d\varphi$.

It follows from H\"older's inequality, \eqref{eq:finale} and Lemma \ref{BG-lem2-techn} that
\begin{align*}
&\int_{B_r}k^{2+2\beta}(w)\phi^2(w)\,dw=C_N\int_0^rk^{2+2\beta}(s)\phi^2(s)s^{2N+1}ds\\
&\quad\leq \left(C_N \int_{0}^r k^p(s)\phi^2(s)s^{2N+1} ds\right)^{2/p}
    \left( C_N\int_{0}^rk^2(s) \phi^2(s) s^{2N+1} ds\right)^{\beta}\\
&\quad= \left(C_N \int_{0}^r k^p(s)s^{2N-2\alpha+1} ds\right)^{2/p}
    \left( \int_{B_r}k^2 \phi^2\right)^{\beta}\\
&\quad \leq \hat{ C}C_N^{\frac{2}{p}-1}C_N \int_{0}^r\left( |k'(s)|^2+k^2(s) \right)s^{2N-2\alpha+1}ds
\left( \int_{B_r}k^2\phi^2 \right)^\beta\\
&\quad = \hat{\hat C} \left( \int_{B_r}|\nablah k|^2\phi^2+\int_{B_r}k^2\phi^2 \right)
    \left( \int_{B_r}k^2\phi^2 \right)^{\beta}.
\end{align*}
This completes the proof.
\end{proof}

We conclude this Appendix by proving the following perturbation result.
\begin{proposition}\label{pr:Bddper}
Assume that the problem \eqref{eq:pb4} has a solution for some $u_0\geq 0$, $f\geq 0$ and let $B\in L^\infty (\HH^N)$.
Then the problem
\begin{equation}\label{eq:pb4B}
 \left\{
 \begin{array}{ll}
 \dfrac{\partial u}{\partial t}=\Deltah u+V(\cdot)u+B(\cdot)u+f & \text{ in }{\mathcal D}'  (\HH^N\times (0,T)),\\ \\
  {\rm esslim}_{t\to 0^+}\displaystyle\int_{\HH^N}u(w,t)\psi(w)\,dw=\int_{\HH^N} u_0(w)\psi (w)\,dw&
  		\text{ } \forall\  \psi\in {\mathcal D}(\HH^N), \\ \\
  u\geq 0 &\text{ on } \HH^N\times(0,T),\\ \\
  Vu\in L^1_{loc}(\HH^N\times(0,T)),
 \end{array}
 \right.
\end{equation}
has a solution.
\end{proposition}
\begin{proof}
Let $u_n$ be the
solution of  \eqref{eq:pb-approx}. We know that $u_n\uparrow u$, $u$ being a solution of \eqref{eq:pb4}.
Suppose that $v_n$ solves
\begin{equation}\label{eq:pb-approxB}
 \left\{
 \begin{array}{ll}
 \dfrac{\partial v_n}{\partial t}=\Deltah v_n+\left(V_n(\cdot)+B(\cdot)\right)v_n+f_n & \text{ in }{\mathcal D}'_T,\\
 \\
  {\rm lim}_{t\to 0^+}\displaystyle\int_{\HH^N}v_n(w,t)\psi(w)\,dw=\displaystyle\int_{\HH^N} u_0(w)\psi (w)\,dw,&\forall \psi\in {\mathcal D}(\HH^N), \\
 \end{array}
 \right.
\end{equation}
where $f_n=\min\{f,n\}$ and $V_n=\min\{V,n\}$. Fix $\lambda\geq \|B\|_\infty$, and consider
\[
 U_n=e^{\lambda t}u_n.
\]
So  $U_n$ satisfies
\[
 \frac{\partial U_n}{\partial t}=\Deltah U_n+(V_n+\lambda)U_n+e^tf_n.
\]
By the Maximum Principle we have
\[
 v_n(w,t)\leq U_n(w,t)\leq e^{\lambda t}u(w,t) \,\hbox{\ for a.e. } (w,t)\in \HH^N\times (0,T).
\]
Clearly  $\{v_n\}$  is an increasing sequence  and since  $u, Vu\in L^1_{loc}(\HH^N\times (0,T))$, it follows by the Monotone Convergence
theorem that $v_n\uparrow v$ and $(V_n+B)v_n\uparrow (V+B)v$ in $L^1_{loc}(\HH^N\times (0,T))$, and $v$ gives a solution of
\eqref{eq:pb4B}.
\end{proof}

\section*{Acknowledgement} A.E. Kogoj, A. Rhandi and C. Tacelli have been partially supported by the Gruppo Nazionale per l'Analisi Matematica, la Probabilit\`a e le loro Applicazioni (GNAMPA) of the Istituto Nazionale di Alta Matematica (INdAM).

\bibliographystyle{amsplain}


\begin{thebibliography}{1}

\bibitem{AGG}
W.~ Arendt, G.R.~ Goldstein, and J.A.~ Goldstein,
{\em Outgrowths of Hardy's inequality}, Contemporary Math. {\bf 412}
(2006), 51--68.

\bibitem{Ar68} D.G.~ Aronson, {\em Non-negative solutions of linear parabolic equations}, Ann. Scuola Norm Sup. Pisa {\bf 22} (1968), 607--694.

\bibitem{BG84}
P.~ Baras and J.A.~ Goldstein, {\em The heat equation with
singular potential}, Trans. Amer. Math. Soc. {\bf 284} (1984), 121--139.

\bibitem{BLU}
A.~Bonfiglioli, F.~Lanconelli and F.~Uguzzoni,
``Stratified Lie Groups and Potential Theory for their Sub-Laplacians",
	Springer Monographs in Mathematics, Berlin, 2007.
\bibitem{CM}
X.~ Cabr\'e and Y.~ Martel,  {\em Existence versus explosion instantan\'ee pour des \'equations de la chaleur lin\'eaires avec potentiel singulier,} C. R. Acad. Sci. Paris, \textbf{329} (1999), 973--978.

\bibitem{CGRT} A. Canale, F. Gregorio, A. Rhandi, C. Tacelli, \emph{Weighted Hardy's inequalities and Kolmogorov-type operators}, Appl. Anal. {\bf 98} (2019), 1236-1254.

\bibitem{DR}
T. ~Durante and A.~ Rhandi, \emph{On the essential self-adjointness of {O}rnstein-{U}hlenbeck operators perturbed by inverse-square potentials}, Discrete Contin. Dyn. Syst. Ser. S {\bf 6} (2013), 649--655.
	
\bibitem{folland_stein_1974}
G.B.~ Folland and E.M.~ Stein, \emph{Estimates for the {$\bar \partial _{b}$}
  complex and analysis on the {H}eisenberg group}, Comm. Pure Appl. Math.
  \textbf{27} (1974), 429--522. 

\bibitem{FR}
S.~ Fornaro and A.~ Rhandi, \emph{On the Ornstein Uhlenbeck operator perturbed by singular potentials in $L^p$-spaces}, Discrete Contin. Dyn. Syst. {\bf 33} (2013), 5049--5058.

\bibitem{FGR}
S.~ Fornaro, F.~ Gregorio and A.~ Rhandi, \emph{Elliptic operators with unbounded diffusion coefficients perturbed by inverse square potentials in Lp-spaces}, Commun. Pure Appl. Anal. {\bf 15} (2016), 2357--2372.

\bibitem{garofalo_lanconelli_1990}
N.~Garofalo and E.~Lanconelli, \emph{Frequency functions on the {H}eisenberg
  group, the uncertainty principle and unique continuation}, Ann. Inst. Fourier
  (Grenoble) \textbf{40} (1990), 313--356. 

\bibitem{GGR}
G.R.~ Goldstein, J.A.~ Goldstein and A.~ Rhandi, \emph{Weighted {H}ardy's inequality and the {K}olmogorov equation perturbed by an inverse-square potential}, Appl. Anal. {\bf 91} (2012), 2057--2071.

\bibitem{GHR}
J.A. Goldstein, D. Hauer and A. Rhandi, \emph{Existence and nonexistence of positive solutions of $p$-Kolmogorov equations perturbed by a Hardy potential}, Nonlinear Anal. {\bf 131} (2016), 121--154.

\bibitem{GK}
J.A.~ Goldstein and I.~ Kombe, {\em Instantaneous blow up},
Contemporary Math. {\bf 327} (2003), 141--150.

\bibitem{GZ1}
J.A.~ Goldstein and Q.S.~ Zhang, {\em On a degenerate heat
equation with a singular potential}, J. Funct. Anal. {\bf 186}
(2001), 342--359.

\bibitem{GZ2}
J.A.~ Goldstein and Q.S.~ Zhang,  {\em Linear parabolic
equations with strong singular potentials}, Trans. Amer.
Math. Soc. {\bf 355} (2003), 197--211.

\bibitem{HR}
 D.~Hauer  and A.~Rhandi, \emph{A weighted {H}ardy inequality and nonexistence of positive solutions}, Arch. Math. {\bf 100} (2013), 273--287.

\bibitem{howe}
R.~Howe, \emph{On the role of the {H}eisenberg group in harmonic analysis},
  Bull. Amer. Math. Soc. (N.S.) \textbf{3} (1980), 821--843. 

\bibitem{jerison_lee_1988}
D.~ Jerison and J.M.~ Lee, \emph{Extremals for the {S}obolev inequality on
  the {H}eisenberg group and the {CR} {Y}amabe problem}, J. Amer. Math. Soc.
  \textbf{1} (1988), 1--13. 

\bibitem{jerison_lee_1989}
D.~ Jerison and J.M.~ Lee, \emph{Intrinsic {CR} normal coordinates and the {CR} {Y}amabe
  problem}, J. Differential Geom. \textbf{29} (1989), 303--343.

\bibitem{semmes}
S.~Semmes, \emph{An introduction to {H}eisenberg groups in analysis and
  geometry}, Notices Amer. Math. Soc. \textbf{50} (2003), 640--646.
\bibitem{VSC}
N.Th.~ Varopoulos, L.~ Saloff-Coste and T.~ Coulhon,
``Analysis and Geometry on Groups", Cambridge University Press, 1992.
\end{thebibliography}

\end{document}